\documentclass[oneside,10pt]{article}          % please do not change
\usepackage[b5paper]{geometry}	    % your paper can be easily printed on a4 or letter paper with enlargenment      
\usepackage{amsfonts,amsmath,latexsym,amssymb} % these packages are required
 \usepackage{hyperref} 
\usepackage{graphicx}
\usepackage{rotating}
\usepackage{rotating}
\usepackage{caption}
\usepackage{multirow}
\usepackage{txfonts}
\usepackage{mathdots}
\usepackage[classicReIm]{kpfonts}
\usepackage{array}
\usepackage{theorem}                % please use one of this two options for theorems
\usepackage{mathrsfs,upref}         % not so essential, but part of journal style
\usepackage{mathptmx}		    % Journal is printed with poscript fonts: 	    % Journal is printed with poscript fonts: 
\usepackage{fdc}	            % Journal style, comment only if you find a bug
\newtheorem{theorem}{Theorem}           
               
\newtheorem{corollary}{Corollary}

\theoremstyle{definition}

\newtheorem{definition}{Definition}
\newtheorem{example}{Example}

\begin{document}

\title[Orthogonal hybrid functions]{Piecewise linear approximate solution of fractional order non-stiff and stiff differential-algebraic equations by orthogonal hybrid functions}

% Short title is optional, it will appear in running heads.
% It is necessary only if the title is to long to be used in running heads

\author{Seshu Kumar Damarla and Madhusree Kundu}

\address{Seshu Kumar Damarla, Department of Chemical Engineering, C.V. Raman College of Engineering, Bhubaneswar-752054, Odisha, India,\\
\email{seshukumar.damarla@gmail.com}}

\address{Madhusree Kundu, Department of Chemical Engineering, National Institute of Technology, Rourkela-769008, Odisha, India,\\
\email{mkundu@nitrkl.ac.in}}

\CorrespondingAuthor{Seshu Kumar Damarla}

%\dedicated{Dedicated to...}                    % Optional

\date{31.01.2018}                               % Please, write the date of submission


        % AMS-2010 subj class. The list can be found on http://www.ams.org/mathscinet/msc/msc2010.html

%\thanks{This research is supported by ...\par This paper is a lecture that was given at...} 
        % Optional. Only one command thanks is allowed, use \par inside text if you need multiple thanks.       

\begin{abstract}
        A simple yet effective numerical method using orthogonal hybrid functions consisting of piecewise constant orthogonal sample-and-hold functions and piecewise linear orthogonal triangular functions is proposed to solve numerically fractional order non-stiff and stiff differential-algebraic equations. The complementary generalized one-shot operational matrices, which are the foundation for the developed numerical method, are derived to estimate the Riemann-Liouville fractional order integral in the new orthogonal hybrid function domain. It is theoretically and numerically shown that the numerical method converges the approximate solutions to the exact solution in the limit of step size tends to zero. Numerical examples are solved using the proposed method and the obtained results are compared with the results of some popular semi-analytical techniques used for solving fractional order differential-algebraic equations in the literature. Our results are in good accordance with the results of those semi-analytical methods in case of non-stiff problems and our method provides valid approximate solution to stiff problem (fractional order version of Chemical Akzo Nobel problem) which those semi-analytical methods failsto solve.
\end{abstract}

\maketitle

%%%%% END OF TITLE PAGE %%%%%%%%%%%%%%%%%%%%%%%%%%%%%%%%%%%%%%%%%%%%%%

%%%%% BODY OF THE PAPER %%%%%%%%%%%%%%%%%%%%%%%%%%%%%%%%%%%%%%%%%%%%%%
% You should eventually delete (after reading) the rest of the text below %%

\section{Introduction}\label{sec1}
In this paper, we solve the fractional order differential-algebraic equations of the following form
\begin{equation} \label{GrindEQ__1_} 
{}_{0}^{C} D_{t}^{\alpha } y_{i} \left(t\right)=f_{i} \left(t,y_{1} \left(t\right),y_{2} \left(t\right),\ldots ,y_{n} \left(t\right)\right), i=1,2,\ldots ,n-1, n\in Z^{+} ,            
\end{equation} 
\begin{equation} \label{GrindEQ__2_} 
0=g_{i}\left(t,y_{1} \left(t\right),y_{2} \left(t\right),\ldots y_{n} \left(t\right)\right), \alpha \in \left(0,1\right], t\in \left[0,1\right].              
\end{equation} 
Here $y_{i} \left(t\right)$ is the $i^{th} $ unknown function, $f_{i} $, $g_{i} $ can be linear or nonlinear functions, ${}_{0}^{C} D_{t}^{\alpha } $ is the Caputo fractional order derivative \cite{ref1}
\begin{equation} \label{GrindEQ__3_} 
{}_{0}^{C} D_{t}^{\alpha } f\left(t\right)=J^{1-\alpha } D^{1} f\left(t\right)=\frac{1}{\Gamma \left(1-\alpha \right)} \int _{0}^{t}\left(t-\tau \right)^{1-\alpha -1} f^{\left(1\right)} \left(\tau \right)d\tau  ,             
\end{equation} 
where $J^{\alpha } f\left(t\right)$ is Riemann-Liouville fractional order integral, \\
$J^{\alpha } f\left(t\right)=\frac{1}{\Gamma \left(\alpha \right)} \int _{0}^{t}\left(t-\tau \right)^{\alpha -1}  f\left(\tau \right)d\tau $.

\noindent The fractional order differential-algebraic equation is a special class of fractional differential equations. The fractional differential equations are the ordinary differential equations involving integrals and/or derivatives of arbitrary order. The subject which deals with the theory of arbitrary order differentiation and arbitrary order integration is fractional calculus \cite{ref2, ref3}. The applications of the subject can be broadly categorized into modelling physical phenomena \cite{ref4}-\cite{ref13} and fractional order controllers \cite{ref14}-\cite{ref16}. For many decades, the subject; fractional calculus did not receive much attention as there were no analytical and numerical methods available to analyze the physical processes, which exhibit fractional order behavior, modelled by the fractional calculus concepts. Hence, devising analytical and numerical methods to solve fractional order integral equations, fractional order integro-differential equations, fractional order ordinary and partial differential equations and fractional order differential-algebraic equations has been an active research area. Abdelkawy et al. (2015) developed a numerical technique based on shifted Jacobi polynomials and spectral collocation method to solve Abel's integral equations of first kind \cite{ref17}. Agarwal et al. (2015) solved fractional Volterra integral equations and non-homogeneous time fractional heat equation using $P_{\alpha} $-transform (integral transform of pathway type) \cite{ref18}. Legendre wavelets have been used by Yi et al. (2016) for numerical solution of fractional order integro-differential equations with weakly singular kernels \cite{ref19}. Using Chebyshev polynomials with spectral tau method, a direct solution technique has been developed in \cite{ref20} to solve multi-order fractional differential equations. Unlike fractional order integral equations, fractional order integro-differential equations and fractional order differential equations, the fractional order partial differential equations are more problematic and more efficient analytical and numerical methods are needed to solve them. In this regard, Fu et al. (2013) proposed a solution procedure to solve time fractional diffusion equations \cite{ref21}. They have used the Laplace transform to convert the fractional diffusion equation into time-independent inhomogeneous equation and then employed truly boundary-only meshless boundary particle method to solve the obtained inhomogeneous equation. In \cite{ref22}, the Kansa method has been used for the first time in the solution of fractional diffusion equation. Efforts are continuously made in proposing analytical, semi-analytical and numerical techniques for solving time fractional order partial differential equations \cite{ref23}-\cite{ref27}. Formulating mathematical models for some real processes which are memory or history based and which calls for the use of fractional derivative and/or fractional integral while mathematically describing them naturally leads to the appearance of fractional order differential-algebraic equations. Unless those fractional order differential-algebraic equations are either analytically or numerically solvable, there is no another way except performing experiments for deep perception of those physical processes. It is very uncommon that the all sorts of fractional order differential-algebraic equations bear analytical solutions. Unsurprisingly many pure and applied mathematicians have been motivated by this inevitable hurdle to wider the range of applicability of the existing semi-analytical and numerical techniques to solve numerically fractional order differential-algebraic equations \cite{ref28}-\cite{ref33}. As we know that none of the numerical methods can solve all categories of fractional differential-algebraic equations, so there is a constant need for more accurate and computationally effective numerical methods which work for most fractional order differential-algebraic equations. 

\noindent Deb et al. \cite{ref34, ref35} proposed the orthogonal hybrid functions (HFs), which are actually a linear combination of the piecewise constant orthogonal sample-and-hold functions and the piecewise linear orthogonal right-handed triangular functions, to find the numerical solution of the linear ordinary differential equations. The orthogonal HFs were further applied for time-invariant, time-varying, delay and delay-free system analysis and identification \cite{ref36}. Realizing the power of the orthogonal HFs, we aim to extend the application of the orthogonal HFs to the fractional order differential-algebraic equations. We accomplish the objective in two steps. The first step is to find a highly accurate approximation by using the orthogonal HFs for the Riemann-Liouville fractional order integral and the second step encompasses the development of the numerical method using the derived HFs estimate for fractional order integral. The proposed numerical method does not require the computation of fractional integrals or fractional derivatives and recursive relations, thus greatly reducing CPU usage, and transforms the given fractional order differential-algebraic equation into a system of algebraic equations which can be solved with minimum effort. The remaining part of the paper is arranged in the following way. Section \ref{sec2} gives a brief introduction to the new orthogonal hybrid functions and their properties. The important result of the paper is given in Section \ref{sec3}. Based on the result of Section \ref{sec3}, an elegant numerical method is developed in Section \ref{sec4}. Section \ref{sec5} studies the convergence of the HF approximate solution of the fractional order differential-algebraic equations. A set of fractional order differential-algebraic equations are solved by the proposed numerical method in Section \ref{sec6}. Section \ref{sec7} presents the concluding remarks. 
%%%%%%%%%%%%%%%%%%%%%%%%%%%%%%%%%%%%%
\section{A brief review of orthogonal hybrid functions}\label{sec2}
\begin{definition}\label{def1}
 Let $S_{i} \left(t\right)$ and $T_{i} \left(t\right)$ be the $i^{th} $ component of the set of piecewise constant sample-and-hold functions, $S_{\left(m\right)} \left(t\right)$, and the set of piecewise linear right-handed triangular functions, $T_{\left(m\right)} \left(t\right)$, respectively, and be defined as    
\end{definition}           
\begin{equation} \label{GrindEQ__4_} 
S_{i} \left(t\right)=\left\{\begin{array}{cc} {1,} & {i{\rm f\; }t\in \left[ih,\left(i+1\right)h\right),} \\ {0,} & {{\rm otherwise,}} \end{array}\right. , T_{i} \left(t\right)=\left\{\begin{array}{cc} {\left(\frac{t-ih}{h} \right),} & {i{\rm f\; }t\in \left[ih,\left(i+1\right)h\right),} \\ {0,} & {{\rm otherwise,}} \end{array}\right. ,            
\end{equation} 
where $i\in \left[0,m-1\right]$, $m$ is the number of subintervals of the interval $t\in \left[0,T\right]$, $h=T/m$.        
\begin{definition}\label{def2}
The $i^{th} $ orthogonal hybrid function (HF) is defined as
\end{definition}
\begin{equation}\label{GrindEQ__5_}
H_{i} \left(t\right)=c_{i} S_{i} \left(t\right)+d_{i} T_{i} \left(t\right),
\end{equation}
where $c_{i} $ and $d_{i} $ are real arbitrary constants.
\begin{definition}\label{def3}
Let us consider a time function, $f\left(t\right)$, of Lebesgue measure, which is defined on $t\in \left[0,T\right]$. The approximation for $f\left(t\right)$ in the orthogonal HF domain is derived as
\end{definition}
\begin{equation} \label{GrindEQ__6_} 
f\left(t\right)\approx \sum _{i=0}^{m-1}\left(c_{i} S_{i} \left(t\right)+d_{i} T_{i} \left(t\right)\right) =C_{S}^{T} S_{\left(m\right)} \left(t\right)+C_{T}^{T} T_{\left(m\right)} \left(t\right),              
\end{equation}
where $C_{S}^{T} =\left[\begin{array}{ccccc} {c_{0} } & {c_{1} } & {c_{2} } & {\cdots } & {c_{m-1} } \end{array}\right]$, $C_{T}^{T} =\left[\begin{array}{ccccc} {d_{0} } & {d_{1} } & {d_{2} } & {\cdots } & {d_{m-1} } \end{array}\right]$, \\
$S_{\left(m\right)} \left(t\right)=\left[\begin{array}{ccccc} {S_{0} \left(t\right)} & {S_{1} \left(t\right)} & {S_{2} \left(t\right)} & {\cdots } & {S_{m-1} \left(t\right)} \end{array}\right]^{Tp}$, $c_{i} =f\left(ih\right)$, $d_{i} =c_{i+1} -c_{i} $,\\
$T_{\left(m\right)} \left(t\right)=\left[\begin{array}{ccccc} {T_{0} \left(t\right)} & {T_{1} \left(t\right)} & {T_{2} \left(t\right)} & {\cdots } & {T_{m-1} \left(t\right)} \end{array}\right]^{Tp}$, $\left[\cdots \right]^{Tp} $ implies transpose. 
\begin{definition}\label{def4}
Let $P_{1ss\left(m\right)} $, $P_{1st\left(m\right)} $, $P_{1ts\left(m\right)} $ and $P_{1tt\left(m\right)} $ be the complementary one-shot operational square matrices of size $m\times m$. The HF estimate for the first order integral of $f\left(t\right)$ is
\end{definition}
\begin{equation} \label{GrindEQ__7_} 
\int _{0}^{t}f\left(\tau \right)d\tau \approx \left(P_{1ss\left(m\right)} C_{S}^{T} +P_{1ts\left(m\right)} C_{T}^{T} \right) S_{\left(m\right)} \left(t\right)+\left(P_{1st\left(m\right)} C_{S}^{T} +P_{1tt\left(m\right)} C_{T}^{T} \right)T_{\left(m\right)} \left(t\right),           
\end{equation} 
where $P_{1ss\left(m\right)} =h\left[\left[\left. \left. \begin{array}{ccccc} {0} & {1} & {1} & {\cdots } & {1} \end{array}\right]\right]\right. \right. $, $P_{1st\left(m\right)} =h\left[\left[\left. \left. \begin{array}{ccccc} {1} & {0} & {0} & {\cdots } & {0} \end{array}\right]\right]\right. \right. $, \\
$P_{1ts\left(m\right)} =\frac{h}{2} \left[\left[\left. \left. \begin{array}{ccccc} {0} & {1} & {1} & {\cdots } & {1} \end{array}\right]\right]\right. \right. $, $P_{1tt\left(m\right)} =\frac{h}{2} \left[\left[\left. \left. \begin{array}{ccccc} {1} & {0} & {0} & {\cdots } & {0} \end{array}\right]\right]\right. \right. $, 
$\left[\left[\left. \left. \begin{array}{ccc} {a} & {b} & {c} \end{array}\right]\right]\right. \right. =\left[\begin{array}{ccc} {a} & {b} & {c} \\ {0} & {a} & {b} \\ {0} & {0} & {a} \end{array}\right]$.\\
The following are some useful properties of orthogonal HFs which bring the ability to HFs to solve the fractional order differential-algebraic equations \cite{ref36}.\\
The piecewise constant sample-and-hold functions and the piecewise linear right-handed triangular functions are orthogonal, for $i,j\in \left[0,m-1\right] $,
\begin{equation} \label{GrindEQ__8_} 
\int _{0}^{T}S_{i} \left(t\right)S_{j} \left(t\right)dt=\left\{\begin{array}{cc} {h,} & {i{\rm f\; }i=j,} \\ {0,} & {{\rm if\; }i\ne j,} \end{array}\right.   \int _{0}^{T}T_{i} \left(t\right)T_{j} \left(t\right)dt=\left\{\begin{array}{cc} {\frac{h}{3} ,} & {i{\rm f\; }i=j,} \\ {\frac{h}{6} ,} & {{\rm if\; }i\ne j.} \end{array}\right.   
\end{equation} 
The components of $S_{\left(m\right)} \left(t\right)$ and $T_{\left(m\right)} \left(t\right)$ are mutually disconnected,
\begin{equation} \label{GrindEQ__9_} 
S_{i} \left(t\right)S_{j} \left(t\right)=\left\{\begin{array}{cc} {S_{i} \left(t\right),} & {i{\rm f\; }i=j,} \\ {0,} & {{\rm if\; }i\ne j,} \end{array}\right.  T_{i} \left(t\right)T_{j} \left(t\right)=\left\{\begin{array}{cc} {T_{i} \left(t\right),} & {i{\rm f\; }i=j,} \\ {0,} & {{\rm if\; }i\ne j,} \end{array}\right.
\end{equation} 
The product $S_{i} \left(t\right)T_{j} \left(t\right)$ can be expanded into the orthogonal HFs as
\begin{equation} \label{GrindEQ__10_} 
S_{i} \left(t\right)T_{j} \left(t\right)=\left\{\begin{array}{cc} {T_{i} \left(t\right),} & {i{\rm f\; }i=j,} \\ {0,} & {{\rm if\; }i\ne j,} \end{array}\right.  i,j\in \left[0,m-1\right]. 
\end{equation} 
A function $g\left(t\right)=\left(h\left(t\right)\right)^{n} $, where $n$ can be an integer or a non-integer, is approximated by means of orthogonal HFs as follows.
\begin{equation} \label{GrindEQ__11_} 
g\left(t\right)=\left[\begin{array}{cccc} {c_{0} } & {c_{1} } & {\cdots } & {c_{m-1} } \end{array}\right]S_{\left(m\right)} \left(t\right)+\left[\begin{array}{cccc} {d_{0} } & {d_{1} }  & {\cdots } & {d_{m-1} } \end{array}\right]T_{\left(m\right)} \left(t\right),     
\end{equation} 
where $c_{j} =\left(h\left(jh\right)\right)^{n} $, $d_{j} =c_{j+1} -c_{j} $, $j=0,1,2,\ldots \ldots ,m-1$.\\
The set of time functions; $f_{1} \left(t\right),f_{2} \left(t\right),f_{3} \left(t\right),\cdots \cdots ,f_{p-1} \left(t\right),f_{p} \left(t\right)$, where $p$ is an integer, is defined on $\left[0,T\right]$. The function $N\left(t,f_{1} \left(t\right),f_{2} \left(t\right),f_{3} \left(t\right),\ldots \ldots ,f_{p} \left(t\right)\right)$, which can be linear or nonlinear, can be expanded into orthogonal TFs domain as
\begin{equation}\label{GrindEQ__12_}
N\left(t,f_{1} \left(t\right),\ldots ,f_{p} \left(t\right)\right)=\left[\begin{array}{ccc} {c_{0} } & {\cdots } & {c_{m-1} } \end{array}\right]S_{\left(m\right)} \left(t\right)+\left[\begin{array}{ccc} {d_{0} } & {\cdots } & {d_{m-1} } \end{array}\right]T_{\left(m\right)} \left(t\right), 
\end{equation}
where $c_{j} =N\left(jh,f_{1} \left(jh\right),f_{2} \left(jh\right),\ldots ,f_{p} \left(jh\right)\right)$, $d_{j} =c_{j+1} -c_{j} $, $j=0,1,2,\ldots,m-1$.
%%%%%%%%%%%%%%%%%%%%%%%%%%%%%%%%%%%%%%%%%%%%%%%%%%%%%%%%%%%%%%%%%%%%%%%%%%

\section{Generalized one-shot operational matrices for fractional integration of $f\left(t\right)$}\label{sec3}
In this section, we generalize the one-shot operational matrices in \eqref{GrindEQ__7_} to the general case of fractional order integration of $f\left(t\right)$. The generalized one-shot operational matrices are the basis for the numerical method we shall develop in the next section.
\begin{theorem}\label{thm1}
The fractional integral of order $\alpha $ of the set of sample-and-hold functions, $S_{\left(m\right)} \left(t\right)$, is approximated via the orthogonal TFs as
\begin{equation} \label{GrindEQ__13_} 
J^{\alpha } S_{\left(m\right)} \left(t\right)=\frac{1}{\Gamma \left(\alpha \right)} \int _{0}^{t}\left(t-\tau \right)^{\alpha -1} S_{\left(m\right)} \left(\tau \right)d\tau  =P_{\alpha ss\left(m\right)} S_{\left(m\right)} \left(t\right)+P_{\alpha st\left(m\right)} T_{\left(m\right)} \left(t\right),         
\end{equation} 
where\\
 $P_{\alpha ss\left(m\right)} =\frac{h^{\alpha } }{\Gamma \left(\alpha +1\right)} \left[\left[\left. \left. \begin{array}{cccc} {0} & {\varsigma _{1} } & {\cdots } & {\varsigma _{m-1} } \end{array}\right]\right]\right. \right. $, $\varsigma _{k} =\left(k^{\alpha } -\left(k-1\right)^{\alpha } \right)$, $k\in \left[1,m-1\right]$,\\
$P_{\alpha st\left(m\right)} =\frac{h^{\alpha } }{\Gamma \left(\alpha +1\right)} \left[\left[\left. \left. \begin{array}{cccc} {1} & {\xi _{1} } & {\cdots } & {\xi _{m-1} } \end{array}\right]\right]\right. \right. , \xi _{k} =\left(k+1\right)^{\alpha } -2k^{\alpha } +\left(k-1\right)^{\alpha } , k\in \left[1,m-1\right]$.\\
\end{theorem}
\begin{proof}
The fractional integral of order $\alpha $ of $S_{0} \left(t\right)$ is
\begin{equation} \label{GrindEQ__14_} 
J^{\alpha } S_{0} \left(t\right)=\frac{1}{\Gamma \left(\alpha \right)} \int _{0}^{t}\left(t-\tau \right)^{\alpha -1} S_{0} \left(\tau \right)d\tau  =\left\{\begin{array}{cc} {0,} & {{\rm for\; }t=0, j=0}, \\ {\frac{h^{\alpha } }{\Gamma \left(\alpha +1\right)} \left(j^{\alpha } -\left(j-1\right)^{\alpha } \right),} & {{\rm for\; }t>0,{\rm \; }j>0.} \end{array}\right.  
\end{equation} 
\end{proof}
Evaluating the expression in \eqref{GrindEQ__14_} at $j=1,2,\ldots ,m-1$ yields the following coefficients.
\begin{equation} \label{GrindEQ__15_} 
c_{0} =0, c_{1} =\frac{h^{\alpha } }{\Gamma \left(\alpha +1\right)} , c_{2} =\frac{h^{\alpha } }{\Gamma \left(\alpha +1\right)} \left(2^{\alpha } -1^{\alpha } \right),            
\end{equation} 
\begin{equation} \label{GrindEQ__16_} 
c_{j} =\frac{h^{\alpha } }{\Gamma \left(\alpha +1\right)} \left(j^{\alpha } -\left(j-1\right)^{\alpha } \right), j=3,4,\ldots ,m-1.            
\end{equation} 
The difference between the consecutive coefficients,
\begin{equation} \label{GrindEQ__17_} 
d_{0} =c_{1} -c_{0} =\frac{h^{\alpha } }{\Gamma \left(\alpha +1\right)} , d_{1} =c_{2} -c_{1} =\frac{h^{\alpha } }{\Gamma \left(\alpha +1\right)} \left(2^{\alpha } -1^{\alpha } \right)-\frac{h^{\alpha } }{\Gamma \left(\alpha +1\right)} ,         
\end{equation} 
\begin{equation} \label{GrindEQ__18_} 
d_{2} =c_{3} -c_{2} =\frac{h^{\alpha } }{\Gamma \left(\alpha +1\right)} \left(3^{\alpha } -2^{\alpha } \right)-\frac{h^{\alpha } }{\Gamma \left(\alpha +1\right)} \left(2^{\alpha } -1^{\alpha } \right),       
\end{equation} 
\begin{equation} \label{GrindEQ__19_} 
d_{j} =c_{j+1} -c_{j} =\frac{h^{\alpha } }{\Gamma \left(\alpha +1\right)} \left(\left(j+1\right)^{\alpha } -2j^{\alpha } +\left(j-1\right)^{\alpha } \right), j=3,4,\ldots \ldots ,m-1.         
\end{equation} 
We can approximate $J^{\alpha } S_{0} \left(t\right)$ in terms of TFs,
\begin{equation} \label{GrindEQ__20_} 
J^{\alpha } S_{0} \left(t\right)=\left[\begin{array}{cccc} {c_{0} } & {c_{1} } {\cdots } & {c_{m-1} } \end{array}\right]S_{\left(m\right)} \left(t\right)+\left[\begin{array}{cccc} {d_{0} } & {d_{1} } & {\cdots } & {d_{m-1} } \end{array}\right]T_{\left(m\right)} \left(t\right).       
\end{equation} 
We can approximate $J^{\alpha } S_{0} \left(t\right)$ in terms of TFs,
\begin{equation} \label{GrindEQ__20_} 
J^{\alpha } S_{0} \left(t\right)=\left[\begin{array}{cccc} {c_{0} } & {c_{1} } {\cdots } & {c_{m-1} } \end{array}\right]S_{\left(m\right)} \left(t\right)+\left[\begin{array}{cccc} {d_{0} } & {d_{1} } & {\cdots } & {d_{m-1} } \end{array}\right]T_{\left(m\right)} \left(t\right).       
\end{equation}
Substituting the expressions for $c_{i} $ and $d_{i} $ in \eqref{GrindEQ__20_},
\begin{equation} \label{GrindEQ__21_} 
J^{\alpha} S_{0} \left(t\right)=\frac{h^{\alpha } }{\Gamma \left(\alpha +1\right)} A1 S_{\left(m\right)} \left(t\right)+ {\frac{h^{\alpha } }{\Gamma \left(\alpha +1\right)} B1 T_{\left(m\right)} \left(t\right)},
\end{equation} 
where $A1=\left[\begin{array}{ccccccc} {0} & {1} & {\left(2^{\alpha } -1\right)} & {\cdots } & {\left(j^{\alpha } -\left(j-1\right)^{\alpha } \right)} & {\cdots } & {\left(\left(m-1\right)^{\alpha } -\left(m-2\right)^{\alpha } \right)} \end{array}\right]$, \\
$B1=\left[\begin{array}{cccccc} {1} & {\left(2^{\alpha } -2\right)} & {\cdots } & {\left(\left(j+1\right)^{\alpha } -2j^{\alpha } +\left(j-1\right)^{\alpha } \right)} & {\cdots } & {B01} \end{array}\right]$,\\
$B10=\left(\left(m\right)^{\alpha } -2\left(m-1\right)^{\alpha } +\left(m-2\right)^{\alpha } \right)$.\\
Rewriting \eqref{GrindEQ__21_},
\begin{equation}\label{GrindEQ__22_}
J^{\alpha } S_{0} \left(t\right)=\frac{h^{\alpha } }{\Gamma \left(\alpha +1\right)} A2 S_{\left(m\right)} \left(t\right)+\frac{h^{\alpha } }{\Gamma \left(\alpha +1\right)} B2 T_{\left(m\right)} \left(t\right),
\end{equation} 
where \\
$A2=\left[\begin{array}{cccccc} {0} & {\varsigma _{1} } & {\varsigma _{2} } & {\varsigma _{3} } & {\cdots } & {\varsigma _{m-1} } \end{array}\right]$, $B2=\left[\begin{array}{cccccc} {1} & {\xi _{1} } & {\xi _{2} } & {\xi _{3} } & {\cdots } & {\xi _{m-1} } \end{array}\right]$, $\varsigma _{k} =\left(k^{\alpha } -\left(k-1\right)^{\alpha } \right)$,\\
 $\xi _{k} =\left(k+1\right)^{\alpha } -2k^{\alpha } +\left(k-1\right)^{\alpha } $, $k\in \left[1,m-1\right]$.\\
 Carrying out fractional integration on the remaining terms and expressing the results via orthogonal TFs,
\begin{equation}\label{GrindEQ__23_}
J^{\alpha } S_{1} \left(t\right)=\frac{h^{\alpha } }{\Gamma \left(\alpha +1\right)} A3 S_{\left(m\right)} \left(t\right)+\frac{h^{\alpha } }{\Gamma \left(\alpha +1\right)} B3 T_{\left(m\right)} \left(t\right), 
\end{equation}
 where $A3=\left[\begin{array}{ccccc} {0} & {0} & {\varsigma _{1} } & {\cdots } & {\varsigma _{m-2} } \end{array}\right]$, $B3=\left[\begin{array}{ccccc} {0} & {1} & {\xi _{1} } & {\cdots } & {\xi _{m-2} } \end{array}\right]$. \\
 \begin{equation*}\begin{array}{l} {\vdots } \\ {\vdots } \end{array}\end{equation*}
\begin{equation} \label{GrindEQ__24_} 
J^{\alpha } S_{m-2} \left(t\right)=\frac{h^{\alpha } }{\Gamma \left(\alpha +1\right)} A4 S_{\left(m\right)} \left(t\right)+\frac{h^{\alpha } }{\Gamma \left(\alpha +1\right)} B4 T_{\left(m\right)} \left(t\right),     
\end{equation} 
where $A4=\left[\begin{array}{ccccc} {0} & {0} & {\cdots } & {0} & {\varsigma _{1} } \end{array}\right]$, $B4=\left[\begin{array}{ccccc} {0} & {\cdots } & {0} & {1} & {\xi _{1} } \end{array}\right]$.
\begin{equation} \label{GrindEQ__25_} 
J^{\alpha } S_{m-1} \left(t\right)=\frac{h^{\alpha } }{\Gamma \left(\alpha +1\right)} A5 S_{\left(m\right)} \left(t\right)+\frac{h^{\alpha } }{\Gamma \left(\alpha +1\right)} B5 T_{\left(m\right)} \left(t\right),    
\end{equation} 
where $A5=\left[\begin{array}{cccccc} {0} & {0} & {\cdots } & {\cdots } & {0} & {0} \end{array}\right]$, $B5=\left[\begin{array}{cccccc} {0} & {\cdots } & {\cdots } & {0} & {0} & {1} \end{array}\right]$.\\
Therefore,
\begin{equation} \label{GrindEQ__26_} 
J^{\alpha } S_{\left(m\right)} \left(t\right)=P_{\alpha ss\left(m\right)} S_{\left(m\right)} \left(t\right)+P_{\alpha st\left(m\right)} T_{\left(m\right)} \left(t\right),             
\end{equation} 
where 
$P_{\alpha ss\left(m\right)} =\frac{h^{\alpha } }{\Gamma \left(\alpha +1\right)} A6$ , $P_{\alpha st\left(m\right)} =\frac{h^{\alpha } }{\Gamma \left(\alpha +1\right)} B6$, \\
$A6=\left[\left[\left. \left. \begin{array}{cccccc} {0} & {\varsigma _{1} } & {\varsigma _{2} } & {\varsigma _{3} } & {\cdots } & {\varsigma _{m-1} } \end{array}\right]\right]\right. \right.$, $B6=\left[\left[\left. \left. \begin{array}{cccccc} {1} & {\xi _{1} } & {\xi _{2} } & {\xi _{3} } & {\cdots } & {\xi _{m-1} } \end{array}\right]\right]\right. \right. $.\\
This proves Theorem \ref{thm1}.
\begin{corollary}\label{corr1}
 If $\alpha =1$, the generalized one-shot operational matrices; $P_{\alpha ss\left(m\right)} $, $P_{\alpha st\left(m\right)} $ in \eqref{GrindEQ__26_} become the one-shot operational matrices, $P_{1ss\left(m\right)} $, $P_{1st\left(m\right)} $, for first order integration of $S_{\left(m\right)} \left(t\right)$.
\end{corollary}
\begin{theorem}\label{thm2}
The HF estimates for the fractional integral of order $\alpha $ of the set of piecewise linear right-handed triangular functions, $T_{\left(m\right)} \left(t\right)$, is
\begin{equation} \label{GrindEQ__27_} 
J^{\alpha } T_{\left(m\right)} \left(t\right)=\frac{1}{\Gamma \left(\alpha \right)} \int _{0}^{t}\left(t-\tau \right)^{\alpha -1} T_{\left(m\right)} \left(\tau \right)d\tau  =P_{\alpha ts\left(m\right)} S_{\left(m\right)} \left(t\right)+P_{\alpha tt\left(m\right)} T_{\left(m\right)} \left(t\right),          
\end{equation} 
where $P_{\alpha ts\left(m\right)} =\frac{h^{\alpha } }{\Gamma \left(\alpha +2\right)} \left[\left[\left. \left. \begin{array}{cccccc} {0} & {\phi _{1} } & {\phi _{2} } & {\phi _{3} } & {\cdots } & {\phi _{m-1} } \end{array}\right]\right]\right. \right. $, \\
$\phi _{k} =k^{\alpha +1} -\left(k-1\right)^{\alpha } \left(k+\alpha \right)$, $k\in \left[1,m-1\right]$, \\
$P_{\alpha tt\left(m\right)} =\frac{h^{\alpha } }{\Gamma \left(\alpha +2\right)} \left[\left[\left. \left. \begin{array}{cccccc} {1} & {\psi _{1} } & {\psi _{2} } & {\psi _{3} } & {\cdots } & {\psi _{m-1} } \end{array}\right]\right]\right. \right.$ , \\
$\psi _{k} =\left(k+1\right)^{\alpha +1} -\left(k+1+\alpha \right)k^{\alpha } -k^{\alpha +1} +\left(k+\alpha \right)\left(k-1\right)^{\alpha } , k\in \left[1,m-1\right]. $
\end{theorem}
\begin{proof}
We get the following expression upon performing fractional integration on $T_{0} \left(t\right)$, 
\begin{equation} \label{GrindEQ__28_} 
J^{\alpha } T_{0} \left(t\right)=\frac{1}{\Gamma \left(\alpha \right)} \int _{0}^{t}\left(t-\tau \right)^{\alpha -1} T_{0} \left(\tau \right)d\tau=\left\{\begin{array}{cc} {0,} & {{\rm for\; }t=0,} \\ {B02,} & {{\rm for\; }t>0,{\rm \; }j\in \left[1,m-1\right].} \end{array}\right.  
\end{equation} 
where $B02={\frac{h^{\alpha } }{\Gamma \left(\alpha +2\right)} \left(j^{\alpha +1} -\left(j-1\right)^{\alpha } \left(j+\alpha \right)\right)}$.\\
In the orthogonal HF domain, $J^{\alpha } T_{0} \left(t\right)$ is expressed as
\begin{equation} \label{GrindEQ__29_} 
J^{\alpha } T_{0} \left(t\right)=\left[\begin{array}{cccc} {c_{0} } & {c_{1} } & {\cdots } & {c_{m-1} } \end{array}\right]S_{\left(m\right)} \left(t\right)+\left[\begin{array}{cccc} {d_{0} } & {d_{1} } & {\cdots } & {d_{m-1} } \end{array}\right]T_{\left(m\right)} \left(t\right),     
\end{equation} 
where 
$c_{j} =\frac{h^{\alpha } }{\Gamma \left(\alpha +2\right)} \left(j^{\alpha +1} -\left(j-1\right)^{\alpha } \left(j+\alpha \right)\right)$, $j=1,2,3,4,\ldots ,m-1$,\\
$d_{j} =c_{j+1} -c_{j} =\frac{h^{\alpha } }{\Gamma \left(\alpha +2\right)} \left(\left(j+1\right)^{\alpha +1} -\left(j+1+\alpha \right)i^{\alpha } -i^{\alpha +1} +\left(j+\alpha \right)\left(j-1\right)^{\alpha } \right)$.\\
Using the expressions for $c_{j} $ and $d_{j} $ and rearranging, 
\begin{equation} \label{GrindEQ__30_} 
J^{\alpha } T_{0} \left(t\right)=\Psi A7 S_{\left(m\right)}\left(t\right) + \Psi B7 T_{\left(m\right)} \left(t\right), 
\end{equation}
where $A7=\left[\begin{array}{cccccc} {0} & {1} & {\cdots } & {\left(j^{\alpha +1} -\left(j-1\right)^{\alpha } \left(j+\alpha \right)\right)} & {\cdots } & {\Omega _{1} } \end{array}\right]$, $\Psi =\frac{h^{\alpha } }{\Gamma \left(\alpha +2\right)} $,\\
 $B7=\left[\begin{array}{ccccc} {1} & {\cdots } & {\left(j+1\right)^{\alpha +1} -\left(j+1+\alpha \right)i^{\alpha } -i^{\alpha +1} +\left(j+\alpha \right)\left(j-1\right)^{\alpha } } & {\cdots } & {\Omega _{2} } \end{array}\right]$,\\
  $\Omega _{1} =\left(\left(m-1\right)^{\alpha +1} -\left(m-2\right)^{\alpha } \left(m-1+\alpha \right)\right)$.\\
$
\Omega _{2} =m^{\alpha +1} -\left(m+\alpha \right)\left(m-1\right)^{\alpha } -\left(m-1\right)^{\alpha +1} +\left(m-1+\alpha \right)\left(m-2\right)^{\alpha } .          
$\\
We now rewrite \eqref{GrindEQ__30_},
\begin{equation}\label{GrindEQ__31_}
J^{\alpha } T_{0} \left(t\right)=\frac{h^{\alpha } }{\Gamma \left(\alpha +2\right)} A8 S_{\left(m\right)} \left(t\right)+\frac{h^{\alpha } }{\Gamma \left(\alpha +2\right)} B8 T_{\left(m\right)} \left(t\right)
\end{equation} 
where $A8=\left[\begin{array}{cccc} {0} & {\phi _{1} } & {\cdots } & {\phi _{m-1} } \end{array}\right]$,
$B8=\left[\begin{array}{cccc} {1} & {\psi _{1} } & {\cdots } & {\psi _{m-1} } \end{array}\right]$, \\
$\phi _{k} =k^{\alpha +1} -\left(k-1\right)^{\alpha } \left(k+\alpha \right)$, \\
$\psi _{k} =\left(k+1\right)^{\alpha +1} -\left(k+1+\alpha \right)k^{\alpha } -k^{\alpha +1} +\left(k+\alpha \right)\left(k-1\right)^{\alpha } $.\\
Following the same procedure, the remaining components of $T_{\left(m\right)} \left(t\right)$ can be fractional integrated and the resulting expressions can be approximated via HFs.
The fractional integral of order $\alpha $ of $T_{\left(m\right)} \left(t\right)$ in HFs domain is
\begin{equation} \label{GrindEQ__32_} 
J^{\alpha } T_{\left(m\right)} \left(t\right)=P_{\alpha ts\left(m\right)} S_{\left(m\right)} \left(t\right)+P_{\alpha tt\left(m\right)} T_{\left(m\right)} \left(t\right),              
\end{equation} 
where
$P_{\alpha ts\left(m\right)} =\frac{h^{\alpha } }{\Gamma \left(\alpha +2\right)} \left[\left[\left. \left. \begin{array}{cccccc} {0} & {\phi _{1} } & {\phi _{2} } & {\phi _{3} } & {\cdots } & {\phi _{m-1} } \end{array}\right]\right]\right. \right.$, \\
$ P_{\alpha tt\left(m\right)} =\frac{h^{\alpha } }{\Gamma \left(\alpha +2\right)} \left[\left[\left. \left. \begin{array}{cccccc} {1} & {\psi _{1} } & {\psi _{2} } & {\psi _{3} } & {\cdots } & {\psi _{m-1} } \end{array}\right]\right]\right. \right.$.\\
This completes the proof.
\end{proof}
\begin{corollary}\label{corr2}
 The complementary pair of one-shot operational matrices; $P_{1ts\left(m\right)} $ and $P_{1tt\left(m\right)} $ for first order integration of $T_{\left(m\right)} \left(t\right)$ in the orthogonal TFs domain can be recovered from the generalized one-shot operational matrices; $P_{\alpha ts\left(m\right)} $ and $P_{\alpha tt\left(m\right)} $ by using $\alpha =1$.
\end{corollary}
\begin{theorem}\label{thm3}
    The formula for approximating the Riemann-Liouville fractional integral of order $\alpha $ of $f\left(t\right)$ by orthogonal TFs is
\begin{equation} \label{GrindEQ__33_} 
J^{\alpha}f\left(t\right) \approx \left(C_{S}^{T} P_{\alpha ss\left(m\right)} +C_{T}^{T} P_{\alpha ts\left(m\right)} \right) S_{\left(m\right)} \left(t\right)+\left(C_{S}^{T} P_{\alpha st\left(m\right)} +C_{T}^{T} P_{\alpha tt\left(m\right)} \right)T_{\left(m\right)} \left(t\right).       
\end{equation}     
\end{theorem}
\begin{proof}
 The Riemann-Liouville fractional order integral of $f\left(t\right)$ is
\begin{equation} \label{GrindEQ__34_} 
J^{\alpha } f\left(t\right)=\frac{1}{\Gamma \left(\alpha \right)} \int _{0}^{t}\left(t-\tau \right)^{\alpha -1} f\left(\tau \right)d\tau =\frac{1}{\Gamma \left(\alpha \right)} \left(t^{\alpha -1} \times f\left(t\right)\right) ,           
\end{equation} 
where the symbol $\times $ is the convolution operator of two functions.\\   
By means of Definition \ref{def3},
\begin{equation}\label{GrindEQ__35_}
J^{\alpha } f\left(t\right) =\frac{1}{\Gamma \left(\alpha \right)} \left(t^{\alpha -1} \times \left(C_{S}^{T} S_{\left(m\right)} \left(t\right)+C_{T}^{T} T_{\left(m\right)} \left(t\right)\right)\right)=C_{S}^{T} J^{\alpha } S_{\left(m\right)} \left(t\right)+C_{T}^{T} J^{\alpha } T_{\left(m\right)} \left(t\right).
\end{equation}
Using Theorems \ref{thm1} and \ref{thm2}, 
\begin{equation} \label{GrindEQ__36_} 
J^{\alpha } f\left(t\right) \approx \left(P_{\alpha ss\left(m\right)} C_{S}^{T} +P_{\alpha ts\left(m\right)} C_{T}^{T} \right) S_{\left(m\right)} \left(t\right)+\left(P_{\alpha st\left(m\right)} C_{S}^{T} +P_{\alpha tt\left(m\right)} C_{T}^{T} \right)T_{\left(m\right)} \left(t\right).       
\end{equation} 
Hence, Theorem \ref{thm3} is proved.
\end{proof}
\begin{corollary}\label{corr3}
 It follows from Corollary \ref{corr1} and Corollary \ref{corr2} that the above expression reduces to that in \eqref{GrindEQ__7_} for$\alpha =1$.
\end{corollary}
%%%%%%%%%%%%%%%%%%%%%%%%%%
\section{Numerical method to solve fractional order differential-algebraic equations}\label{sec4}
We consider the following fractional order differential-algebraic equations.
\begin{equation*}
{}_{0}^{C} D_{t}^{\alpha } y_{i} \left(t\right)=f_{i} \left(t,y_{1} \left(t\right), \ldots ,y_{n} \left(t\right)\right), i=1,2,\ldots, n-1, t\in \left[0,1\right],
\end{equation*}
\begin{equation}\label{GrindEQ__37_} 
0=g\left(t,y_{1} \left(t\right), \ldots ,y_{n} \left(t\right)\right), \alpha \in \left(0,1\right], ,y_{j} \left(0\right)=a_{j} , j=1,2, \ldots, n.        
\end{equation} 
Rewriting \eqref{GrindEQ__37_},
\begin{equation} \label{GrindEQ__38_} 
y_{i} \left(t\right)=y_{i} \left(0\right)+J^{\alpha } f_{i} \left(t,y_{1} \left(t\right),y_{2} \left(t\right),y_{3} \left(t\right),\ldots \ldots ,y_{n} \left(t\right)\right),            
\end{equation} 
\begin{equation} \label{GrindEQ__39_} 
0=g\left(t,y_{1} \left(t\right),y_{2} \left(t\right),y_{3} \left(t\right),\ldots \ldots ,y_{n} \left(t\right)\right).              
\end{equation} 
By utilizing Definition \ref{def3} and Equation \eqref{GrindEQ__12_}, we get
\begin{equation} \label{GrindEQ__40_} 
y_{i} \left(t\right)\approx C_{Si}^{T} S_{\left(m\right)} \left(t\right)+C_{Ti}^{T} T_{\left(m\right)} \left(t\right), y_{i} \left(0\right)=C_{S0i}^{T} S_{\left(m\right)} \left(t\right)+C_{T0i}^{T} T_{\left(m\right)} \left(t\right).        
\end{equation} 
\begin{equation} \label{GrindEQ__41_} 
f_{i} \left(t,y_{1} \left(t\right),y_{2} \left(t\right),y_{3} \left(t\right),\ldots \ldots ,y_{n} \left(t\right)\right)\approx \tilde{C}_{Si}^{T} S_{\left(m\right)} \left(t\right)+\tilde{C}_{Ti}^{T} T_{\left(m\right)} \left(t\right).         
\end{equation} 
\begin{equation} \label{GrindEQ__42_} 
g\left(t,y_{1} \left(t\right),y_{2} \left(t\right),y_{3} \left(t\right),\ldots \ldots ,y_{n} \left(t\right)\right)\approx C_{Sg}^{T} S_{\left(m\right)} \left(t\right)+C_{Tg}^{T} T_{\left(m\right)} \left(t\right).           
\end{equation} 
Equations \eqref{GrindEQ__38_} and \eqref{GrindEQ__39_} become,
\begin{equation} \label{GrindEQ__43_} 
C_{Si}^{T} S_{\left(m\right)} \left(t\right)+C_{Ti}^{T} T_{\left(m\right)} \left(t\right)=C_{S0i}^{T} S_{\left(m\right)} \left(t\right)+C_{T0i}^{T} T_{\left(m\right)} \left(t\right)+J^{\alpha } \left(\tilde{C}_{Si}^{T} S_{\left(m\right)} \left(t\right)+\tilde{C}_{Ti}^{T} T_{\left(m\right)} \left(t\right)\right),        
\end{equation} 
\begin{equation} \label{GrindEQ__44_} 
0=C_{Sg}^{T} S_{\left(m\right)} \left(t\right)+C_{Tg}^{T} T_{\left(m\right)} \left(t\right).               
\end{equation} 
Using Theorem \ref{thm3}, 
\begin{equation} \label{GrindEQ__45_} 
C_{Si}^{T} S_{\left(m\right)} \left(t\right)+C_{Ti}^{T} T_{\left(m\right)} \left(t\right)= C_{S0i}^{T} S_{\left(m\right)} \left(t\right)+C_{T0i}^{T} T_{\left(m\right)} \left(t\right)+B03 S_{\left(m\right)} \left(t\right)+ B04 T_{\left(m\right)} \left(t\right),       
\end{equation} 
\begin{equation} \label{GrindEQ__46_} 
0=C_{Sg}^{T} S_{\left(m\right)} \left(t\right)+C_{Tg}^{T} T_{\left(m\right)} \left(t\right).               
\end{equation} 
where $B03=\left(\tilde{C}_{Si}^{T} P_{\alpha ss\left(m\right)} +\tilde{C}_{Ti}^{T} P_{\alpha ts\left(m\right)} \right)$, $B04=\left(\tilde{C}_{Si}^{T} P_{\alpha st\left(m\right)} +\tilde{C}_{Ti}^{T} P_{\alpha tt\left(m\right)} \right)$.\\
Comparing the coefficients of $S_{\left(m\right)} \left(t\right)$ and $T_{\left(m\right)} \left(t\right)$, 
\begin{equation} \label{GrindEQ__47_} 
\left. \begin{array}{l} {C_{Si}^{T} =C_{S0i}^{T} +\left(\tilde{C}_{Si}^{T} P_{\alpha ss\left(m\right)} +\tilde{C}_{Ti}^{T} P_{\alpha ts\left(m\right)} \right)} \\ {C_{Ti}^{T} =C_{T0i}^{T} +\left(\tilde{C}_{Si}^{T} P_{\alpha ts\left(m\right)} +\tilde{C}_{Ti}^{T} P_{\alpha tt\left(m\right)} \right)} \\ {0=C_{Sg}^{T} } \\ {0=C_{Tg}^{T} } \end{array}\right\}.             
\end{equation} 
Solving \eqref{GrindEQ__47_} produces the HF estimate for the $i^{th} $ unknown, $y_{i} \left(t\right)$, $i=1,2,\ldots ,n$.
%%%%%%%%%%%%%%%%%%%%%%%%%%%%%%%%%%%%%%%%%%%%%%%%%%%%%%%%%%%%%%%%%%%%
\section{Convergence analysis}\label{sec5}
Let $\tilde{y}_{i} \left(t\right)$ be the HF estimate for the actual solution, $y_{i} \left(t\right)$, of fractional order differential-algebraic equations in \eqref{GrindEQ__37_}.\\
The error between the approximate solution and the exact solution of \eqref{GrindEQ__37_} is defined on $j^{th} $ subinterval, $\left[jh,\left(j+1\right)h\right)$, as
\begin{equation} \label{GrindEQ__48_} 
\varepsilon _{i} \left(t\right)=\left|y_{i} \left(t\right)-\tilde{y}_{i} \left(t\right)\right|, i\in \left[1,n\right], t\in \left[jh,\left(j+1\right)h\right), j\in \left[0,m-1\right].           
\end{equation}
Using Definitions \ref{def1} and \ref{def3},
\begin{equation} \label{GrindEQ__49_} 
\begin{array}{l} {\varepsilon _{i} \left(t\right)=y_{i} \left(t\right)-\left(y_{i} \left(jh\right)+\left(y_{i} \left(\left(j+1\right)h\right)-y_{i} \left(jh\right)\right)\frac{\left(t-jh\right)}{h} \right),} \\ {{\rm \; \; \; \; \; \; \; \; \; }=y_{i} \left(t\right)-\left(y_{i} \left(jh\right)+\frac{\left(y_{i} \left(\left(j+1\right)h\right)-y_{i} \left(jh\right)\right)}{h} \left(t-jh\right)\right),} \\ {{\rm \; \; \; \; \; \; \; \; \; }=y_{i} \left(t\right)-\left(y_{i} \left(jh\right)+\left(\frac{dy_{i} \left(t\right)}{dt} \right)_{t=jh} \left(t-jh\right)\right).} \end{array} 
\end{equation} 
The Taylor series expansion of $y_{i} \left(t\right)$ with center $jh$ is
\begin{equation} \label{GrindEQ__50_} 
y_{i} \left(t\right)=y_{i} \left(jh\right)+\left(\frac{dy_{i} \left(t\right)}{dt} \right)_{t=jh} \left(t-jh\right)+\left(\frac{d^{2} y_{i} \left(t\right)}{dt^{2} } \right)_{t=jh} \frac{\left(t-jh\right)^{2} }{2!} + B05,       
\end{equation} 
where $B05=\sum _{k=3}^{\infty }\left(\frac{d^{k} y_{i} \left(t\right)}{dt^{k} } \right)_{t=jh} \frac{\left(t-jh\right)^{k} }{k!}$.\\
Considering the second order Taylor series approximation for $y_{i} \left(t\right)$ and employing it in \eqref{GrindEQ__49_},
\begin{equation} \label{GrindEQ__51_} 
\varepsilon _{i} \left(t\right)=\left(\frac{d^{2} y_{i} \left(t\right)}{dt^{2} } \right)_{t=jh} \frac{\left(t-jh\right)^{2} }{2!} =y_{i}^{''} \left(jh\right)\frac{\left(t-jh\right)^{2} }{2!}.  
\end{equation} 
Let us make the following assumption.
\begin{equation} \label{GrindEQ__52_} 
M=\max \left(\left|y_{i}^{''} \left(0\right)\right|,\left|y_{i}^{''} \left(h\right)\right|,\left|y_{i}^{''} \left(2h\right)\right|,\left|y_{i}^{''} \left(3h\right)\right|,\ldots \ldots ,\left|y_{i}^{''} \left(\left(m-1\right)h\right)\right|\right).
\end{equation} 
We now calculate $\left\| \varepsilon _{i} \left(t\right)\right\| _{1} $ on $j^{th} $ subinterval, $\left[jh,\left(j+1\right)h\right)$.
\begin{equation} \label{GrindEQ__53_} 
\left\| \varepsilon _{i} \left(t\right)\right\| _{1} =\int _{jh}^{\left(j+1\right)h}\left|\varepsilon _{i} \left(t\right)\right|dt =\int _{jh}^{\left(j+1\right)h}\left|y_{i}^{''} \left(jh\right)\right|\frac{\left(t-jh\right)^{2} }{2} dtn= B06,        
\end{equation}
where $B06=M\int _{jh}^{\left(j+1\right)h}\frac{\left(t-jh\right)^{2} }{2} dt =\frac{Mh^{3} }{6}$.\\
Let $\varepsilon _{m} \left(t\right)$ be the sum of errors, $\varepsilon _{i} \left(t\right)$.
\begin{equation} \label{GrindEQ__54_} 
\varepsilon _{m} \left(t\right)=\sum _{i=0}^{m-1}\varepsilon _{i} \left(t\right) .                
\end{equation} 
Calculating $\left\| \varepsilon _{m} \left(t\right)\right\| _{1} $,
\begin{equation*}
\left\| \varepsilon _{m} \left(t\right)\right\| _{1} =\int _{0}^{1}\left|\varepsilon _{m} \left(t\right)\right|dt=\int _{0}^{1}\left(\sum _{i=0}^{m-1}\left|\varepsilon _{i} \left(t\right)\right| \right)dt  =\sum _{i=0}^{m-1}\left(\int _{0}^{1}\left|\varepsilon _{i} \left(t\right)\right|dt \right)=B07,
\end{equation*}
where $B07=\sum _{i=0}^{m-1}\left\| \varepsilon _{i} \left(t\right)\right\| _{1} =\frac{mMh^{3} }{6}  =\frac{M}{6m^{2} }$.\\
Taking limit,
\begin{equation} \label{GrindEQ__55_} 
\begin{array}{c} {\lim } \\ {m\to \infty } \end{array}\left\| \varepsilon _{m} \left(t\right)\right\| _{1} =\begin{array}{c} {\lim } \\ {m\to \infty } \end{array}\frac{M}{6m^{2} } =0.              
\end{equation} 
Therefore, 
\begin{equation} \label{GrindEQ__56_} 
\begin{array}{c} {\lim } \\ {m\to \infty } \end{array}\varepsilon _{m} \left(t\right)=0.                
\end{equation} 
The approximate solution, $\tilde{y}_{i} \left(t\right)$, of fractional order differential-algebraic equation obtained by the proposed numerical method converges to the actual solution when sufficiently large number of subintervals are considered.
%%%%%%%%%%%%%%%%%%%%%%%%%%%%%%%%%%%%%%%%%%%%%%%%%%%%%%%%%%%%%
\section{Numerical examples}\label{sec6}
In this section, we shall solve linear and nonlinear differential-algebraic equations of fractional order using the numerical method devised in Section \ref{sec4}.
\begin{example}\label{ex1}
The linear fractional order differential-algebraic equations are
\begin{equation} \label{GrindEQ__57_} 
_{0}^{C} D_{t}^{0.5} x_{1} \left(t\right)+2x_{1} \left(t\right)-\frac{\Gamma \left(3.5\right)}{2} x_{2} \left(t\right)+x_{3} \left(t\right)=2t^{2.5} +\sin t, x_{1} \left(0\right)=0,           
\end{equation} 
\begin{equation} \label{GrindEQ__58_} 
{}_{0}^{C} D_{t}^{0.5} x_{2} \left(t\right)+x_{2} \left(t\right)+x_{3} \left(t\right)=\frac{2}{\Gamma \left(2.5\right)} t^{1.5} +t^{2} +\sin t, t\in \left[0,1\right], x_{2} \left(0\right)=0,         
\end{equation} 
\begin{equation} \label{GrindEQ__59_} 
0=2t^{2.5} +t^{2} -\sin t-\left(2x_{1} \left(t\right)+x_{2} \left(t\right)-x_{3} \left(t\right)\right), x_{3} \left(0\right)=0.     
\end{equation} 
The exact solution of this problem is $x_{1} \left(t\right)=t^{2.5} $, $x_{2} \left(t\right)=t^{2} $, $x_{3} \left(t\right)=\sin t$.
\end{example}
The given fractional order linear differential-algebraic system is solved and $\infty $-norm of the error between the exact solution and the piecewise linear HF approximate solution is computed for various values of $m$ and given along with the respective elapsed times in Table \ref{table1}. As proved theoretically in the preceding section, the HF solution converges to the actual solution as the step size, $h$, decreases. The method produces approximate solution with acceptable accuracy with $h=1/300$ in just 12.425064 seconds, therefore, it is pretty fast.
%%%%%%%%%%%%%%%%
\begin{table}[tbp]
\caption{Maximal absolute errors for Example \ref{ex1}}
\centering
\begin{tabular}{c c c c c}
\hline
\multicolumn{1}{p{0.8cm}}{\centering $m$} & \multicolumn{1}{p{1.5cm}}{\centering $\left\| \varepsilon _{1} \right\| _{\infty } $} & \multicolumn{1}{p{1.5cm}}{\centering $\left\| \varepsilon _{2} \right\| _{\infty } $} & \multicolumn{1}{p{1.5cm}}{\centering $\left\| \varepsilon _{3} \right\| _{\infty } $}  &  \multicolumn{1}{p{1.5cm}}{\centering CPU time (s) } \\
\hline  
 10 &5.754133e-04 &7.429639e-04 &0.001673 &0.212426\\
50  &2.357914e-05 &4.136157e-05 &6.635931e-05 &0.584962\\
100 &5.929472e-06 &1.114655e-05 &1.706158e-05 &1.661539\\
150 &2.642181e-06 &5.117727e-06 &7.699577e-06 &3.334825\\
200 &1.488526e-06 &2.934228e-06 &4.375460e-06 &5.454437\\
250 &9.536572e-07 &1.902277e-06 &2.820419e-06 &8.695571\\
300 &6.627709e-07 &1.333553e-06 &1.969844e-06 &12.425064\\
\hline
\end{tabular}
\label{table1} 
\end{table} 
%%%%%%%%%%%%
\begin{example}\label{ex2}
The fractional order nonlinear differential-algebraic equations are
\begin{equation} \label{GrindEQ__60_} 
{}_{0}^{C} D_{t}^{0.5} x_{1} \left(t\right)+x_{1} \left(t\right)x_{2} \left(t\right)-x_{3} \left(t\right)=\frac{6}{\Gamma \left(3.5\right)} t^{2.5} +2t^{4} +t^{7} -e^{t} -\sin t, t\in \left[0,1\right],          
\end{equation} 
\begin{equation} \label{GrindEQ__61_} 
{}_{0}^{C} D_{t}^{0.5} x_{2} \left(t\right)-\frac{\Gamma \left(5\right)}{\Gamma \left(4.5\right)} t^{0.5} x_{1} \left(t\right)+2x_{2} \left(t\right)+x_{1} \left(t\right)x_{3} \left(t\right)=\frac{2}{\Gamma \left(1.5\right)} t^{0.5} + B08,       
\end{equation} 
\begin{equation} \label{GrindEQ__62_} 
0=e^{t} +t\sin t-2t^{3} -\left(x_{1}^{2} \left(t\right)-x_{2} \left(t\right)t^{2} +x_{3} \left(t\right)\right), x_{1} \left(0\right)=0, x_{2} \left(0\right)=0, x_{3} \left(0\right)=1,       
\end{equation} 
where $B08=4t+2t^{4} +t^{3} e^{t} +t^{4} \sin t$.\\
We have the analytical solution, $x_{1} \left(t\right)=t^{3} $, $x_{2} \left(t\right)=2t+t^{4} $, $x_{3} \left(t\right)=e^{t} +t\sin t$, for the given fractional order differential-algebraic equations.
\end{example}
Table \ref{table2} presents the maximal absolute errors produced by the proposed numerical method with different values of $m$. Since the fractional differential-algebraic equations in \eqref{GrindEQ__60_} to \eqref{GrindEQ__62_} is nonlinear and a bit harder than Example \ref{ex1}, the numerical method needs little higher CPU usage than it required to solve Example \ref{ex1} yet it maintains good accuracy and takes reasonable computational time.
%%%%%%%%%%%%%%
\begin{table}[tbp]
\caption{Performance of the proposed numerical method for Example \ref{ex2}}
\centering
\begin{tabular}{c c c c c}
\hline
\multicolumn{1}{p{1.0cm}}{\centering $m$} & \multicolumn{1}{p{1.5cm}}{\centering $\left\| \varepsilon _{1} \right\| _{\infty } $} & \multicolumn{1}{p{1.5cm}}{\centering $\left\| \varepsilon _{2} \right\| _{\infty } $} & \multicolumn{1}{p{1.5cm}}{\centering $\left\| \varepsilon _{3} \right\| _{\infty } $}  &  \multicolumn{1}{p{1.5cm}}{\centering CPU time (s) } \\
\hline  
10 &0.001372267 &0.020472485 &9.106405e-04 &0.244156\\
50 &5.500109e-05 &0.004986160 &6.182205e-05 &0.811694\\
100 &1.325811e-05 &0.002628090 &2.036216e-05 &2.402352\\
150 &5.705400e-06 &0.001795134 &1.074409e-05 &4.802019\\
200 &3.120540e-06 &0.001366376 &6.853949e-06 &8.209396\\
250 &1.948521e-06 &0.001104311 &4.846044e-06 & 12.625149\\
300 &1.323821e-06 &9.272787e-04 &3.655260e-06 &18.467165\\
\hline
\end{tabular}
\label{table2} 
\end{table} 

%%%%%%%%%%%%
%%%%%%%%%%%%
\begin{table}[tbp]
\caption{Error analysis for Example \ref{ex3}}
\centering
\begin{tabular}{c c c c c}
\hline
\multicolumn{1}{p{1.0cm}}{\centering $m$} & \multicolumn{1}{p{1.5cm}}{\centering $\left\| \varepsilon _{1} \right\| _{\infty } $} & \multicolumn{1}{p{1.5cm}}{\centering $\left\| \varepsilon _{2} \right\| _{\infty } $} & \multicolumn{1}{p{1.5cm}}{\centering $\left\| \varepsilon _{3} \right\| _{\infty } $}  &  \multicolumn{1}{p{1.5cm}}{\centering CPU time (s) } \\
\hline  
10 &0.0022695      &0.0123438     &3.06898e-04  &0.207223\\
50 &9.06163e-05    &4.92649e-04  &1.22631e-05  &0.524952\\
100 & 2.26527e-05 &1.23153e-04  &3.06569e-06   &1.540332\\
150 &1.00677e-05  &5.47341e-05  &1.36252e-06   &3.077976\\
200 &5.66308e-06  &3.07877e-05  &7.66417e-07   &5.316336\\
250 &3.62436e-06  &1.97040e-05   &4.90506e-07  &8.200463\\
300 &2.51690e-06   &1.36833e-05  &3.40629e-07  &11.630141\\
\hline
\end{tabular}
\label{table3} 
\end{table} 
%%%%%%%%%%%%%%%%%%
\begin{example}\label{ex3}
Consider the following fractional order nonlinear differential-algebraic equations
\begin{equation} \label{GrindEQ__63_} 
{}_{0}^{C} D_{t}^{\alpha } x\left(t\right)-x\left(t\right)+z\left(t\right)x\left(t\right)=1, t\in \left[0,1\right],             
\end{equation} 
\begin{equation} \label{GrindEQ__64_} 
{}_{0}^{C} D_{t}^{\alpha } z\left(t\right)-y\left(t\right)+x^{2} \left(t\right)+z\left(t\right)=0, \alpha \in \left(0,1\right],
\end{equation} 
\begin{equation} \label{GrindEQ__65_} 
y\left(t\right)-x^{2} \left(t\right)=0, x\left(0\right)=1, y\left(0\right)=1, z\left(0\right)=1.             
\end{equation} 
This problem has closed form solution, $x\left(t\right)=e^{t} $, $y\left(t\right)=e^{2t} $, $z\left(t\right)=e^{-t} $, when the fractional order , $\alpha $, equals 1.
\end{example}
Table \ref{table3} presents the maximum absolute error between the HF solution and the exact solution of the integer order version of the problem in \eqref{GrindEQ__63_} to \eqref{GrindEQ__65_}. The HF solutions given in Tables \ref{table4} ($\alpha =1$) to \ref{table6} are in accordance with the solutions obtained by homotopy analysis method, Adomian decomposition method, variational iteration method in \cite{ref30} (see Example 5.3 in \cite{ref30}), fractional differential transform method in \cite{ref32} (see Example 2 in \cite{ref32}) and iterative decomposition method in \cite{ref33} (see Example 2 in \cite{ref33}). Comparing with those semi-analytical techniques, the TFs based numerical method exhibited good performance in terms of accuracy and computational speed (Tables \ref{table3} and \ref{table7}).
%%%%%%%%%%%%%%%%%
%%%%%%%%%%%%
\begin{table}[tbp]
\caption{Absolute errors for Example \ref{ex3}$(\alpha=1m m=300)$}
\centering
\begin{tabular}{c c c c}
\hline
\multicolumn{1}{p{1.1cm}}{\centering $t$} &\multicolumn{1}{p{1.9cm}}{\centering $x\left(t\right)$}
& \multicolumn{1}{p{1.9cm}}{\centering $y\left(t\right)$} 
& \multicolumn{1}{p{1.9cm}}{\centering $z\left(t\right)$}\\
\hline
0    &9.9999999e-08  &9.9999999e-08  &1.0000000e-06\\
0.1  &8.999999e-08  &8.99999999e-07  &4.000000e-07\\
0.2  &1.999999e-07  &1.20000000e-06  &0\\
0.3  &2.999999e-07  &1.80000000e-06  &0\\
0.4  &6.000000e-07  &2.49999999e-06  &1.0000000e-06\\
0.5  &8.000000e-07  &3.29999999e-06  &0\\
0.6  &1.000000e-06  &4.60000000e-06  &0\\
0.7  &1.999999e-06  &6.1999999e-06  &9.99999999e-07\\
0.8  &2.499999e-06  &7.6999999e-06  &0\\
0.9  &1.999999e-06  &1.0500000e-05  &0\\
1    &2.499999e-06  &1.3700000e-05  &3.00000000e-07\\
\hline
\end{tabular}
\label{table4} 
\end{table} 

%%%%%%%%%%%%
\begin{table}[tbp]
\caption{HF solution of Example \ref{ex3} ($\alpha=0.5, m=300$)}
\centering
\begin{tabular}{c c c c}
\hline
\multicolumn{1}{p{1.0cm}}{\centering $t$} & \multicolumn{1}{p{1.0cm}}{\centering $x(t)$} &\multicolumn{1}{p{1.5cm}}{\centering $y(t)$} &\multicolumn{1}{p{1.5cm}}{\centering $z(t)$} \\
\hline
0    &0.9999999  &1            &0.9999999 \\  
0.1  &1.4678387  &2.1545505    &0.7235289 \\  
0.2  &1.7411092  &3.0314614    &0.6437595  \\ 
0.3  &1.9927769  &3.9711600    &0.5919981  \\ 
0.4  & 2.2392505 & 5.0142429    & 0.5535905 \\  
0.5  &2.4871417  & 6.1858739    &0.5231438 \\  
0.6  &2.7401229  &7.5082735    &0.4980139 \\  
0.7  &3.0006555  &9.0039339    &0.4766936 \\  
0.8  &3.2706177  &10.696940    &0.4582380  \\ 
0.9  &3.5515825  &12.613738    &0.4420143  \\ 
1    &3.8449601  &14.7837188    &0.4275772 \\  
\hline
\end{tabular}
\label{table5} 
\end{table} 
%%%%%%%%%%%%%%%%%%%%%%%
\begin{table}[tbp]
\caption{HF solution of Example 3 ($\alpha=0.75, m=300$)}
\centering
\begin{tabular}{c c c c}
\hline
\multicolumn{1}{p{1.0cm}}{\centering $t$} & \multicolumn{1}{p{1.0cm}}{\centering $x(t)$} &\multicolumn{1}{p{1.5cm}}{\centering $y(t)$} &\multicolumn{1}{p{1.5cm}}{\centering $z(t)$} \\
\hline
0     &1                 &1                 &1             \\
0.1  &1.2187008    &1.4852318   &0.8282436\\
0.2  &1.4000240    &1.9600672   &0.7325794\\
0.3  &1.5841244    &2.5094501   &   0.6603331\\
0.4  &1.7769073    &3.1573996   &0.6021174\\
0.5  &1.9813863    &3.9258918   &0.5535994\\
0.6  &2.1997456    & 4.8388809   &0.5122823\\
0.7  &2.4338850    &5.9237964   &0.4765525\\
0.8  &2.6856270    & 7.2125927   &0.4452903\\
0.9  & 2.9568156    &8.7427588   &0.4176801\\
1     &3.2493726    & 10.558422   &0.39310661\\
\hline
\end{tabular}
\label{table6} 
\end{table} 
%%%%%%%%%%%%%%%%%%
%%%%%%%%%%%%%%%%%%%%
\begin{table}[tbp]
\caption{CPU time (in seconds) required by our method}
\centering
\begin{tabular}{c c c c c}
\hline
\multicolumn{1}{p{1.0cm}}{\centering $\alpha$} & \multicolumn{1}{p{1.5cm}}{\centering Example \ref{ex3}} & \multicolumn{1}{p{1.5cm}}{\centering Example \ref{ex4}} & \multicolumn{1}{p{1.5cm}}{\centering $\alpha$}  &  \multicolumn{1}{p{1.5cm}}{\centering Example \ref{ex5}} \\
\hline  
1      &11.630141  &11.682081   &1     &82.839503\\
0.75  &11.655617  &11.748390   &0.8   &75.263996\\
0.5    & 12.331451  &14.243483   &0.9   & 81.163712\\
\hline  
\end{tabular}
\label{table7} 
\end{table} 
%%%%%%%%%%%%%%%%%%%%%
\begin{table}[tbp]
\caption{Error analysis for Example \ref{ex4}}
\centering
\begin{tabular}{c c c c c}
\hline
\multicolumn{1}{p{1.0cm}}{\centering $m$ } & \multicolumn{1}{p{1.5cm}}{\centering  $\left\| \varepsilon _{1} \right\| _{\infty } $} & \multicolumn{1}{p{1.5cm}}{\centering $\left\| \varepsilon _{2} \right\| _{\infty }$ } & \multicolumn{1}{p{1.5cm}}{\centering $\left\| \varepsilon _{3} \right\| _{\infty } $}  &  \multicolumn{1}{p{1.5cm}}{\centering CPU time (s)} \\
\hline 
10     & 0.00604678222  &0.01569367613   &0.00569062735    &0.213382 \\
50     &2.3933286e-04    &6.3246700e-04    &2.3315865e-04    &0.538698 \\
100   &5.9813254e-05    &1.5815094e-04    &5.8324203e-05    &1.557116 \\
150   &2.6582116e-05    & 7.0290756e-05   &2.5928395e-05    &3.058464 \\
200   &1.4952174e-05    & 3.9537932e-05   &1.4584885e-05    &5.223770 \\
250   & 9.5693235e-06   &2.5303467e-05    &9.333862e-06      &8.087865 \\
300   &6.6453513e-06    &1.7571084e-05	 & 6.4814216e-06   &11.68208 \\
\hline 
\end{tabular}
\label{table8} 
\end{table} 	
%%%%%%%%%%%%%%%%%%%%%%%%%%%%

\begin{table}[tbp]
\caption{Absolute errors for Example \ref{ex4} ($\alpha=1, m=300$)}
\centering
\begin{tabular}{c c c c}
\hline
\multicolumn{1}{p{1.1cm}}{\centering $t$} &\multicolumn{1}{p{1.9cm}}{\centering $x\left(t\right)$}
& \multicolumn{1}{p{1.9cm}}{\centering $y\left(t\right)$} 
& \multicolumn{1}{p{1.9cm}}{\centering $z\left(t\right)$}\\
\hline
0    &1.50350000e-15      &6.563500000e-13    	&0\\
0.1  &9.99999999e-08     &9.999999999e-08 		&9.9999999e-08\\
0.2  &1.00000000e-07     &4.999999999e-07 		&4.0000000e-07\\
0.3  &1.99999999e-07     &1.199999999e-06 		&1.0999999e-06\\
0.4  &3.00000000e-07     &2.300000000e-06  		&2.0000000e-06\\
0.5  &6.99999999e-07  	  &3.79999999e-06  		&3.2000000e-06\\
0.6  &1.19999999e-06  	  &5.80000000e-06  		&4.5000000e-06\\
0.7  &1.8999999e-06  	  &8.29999999e-06 		&5.5999999e-06\\
0.8  &2.9999999e-06  	 &1.11999999e-05  		&6.4000000e-06\\
0.9  &4.6000000e-06  	 &1.43000000e-05  		&6.2000000e-06\\
1    &6.7000000e-06  	  &1.7500000000-05 		&4.2000000e-06\\
\hline
\end{tabular}
\label{table9} 
\end{table} 
%%%%%%%%%%%%%%%%%
%%%%%%%%%%%%%%%%%%%%%%%
\begin{table}[tbp]
\caption{HF solution of Example 4 ($\alpha=0.5, m=300$)}
\centering
\begin{tabular}{c c c c}
\hline
\multicolumn{1}{p{1.0cm}}{\centering $t$} & \multicolumn{1}{p{1.0cm}}{\centering $x(t)$} &\multicolumn{1}{p{1.5cm}}{\centering $y(t)$} &\multicolumn{1}{p{1.5cm}}{\centering $z(t)$} \\
\hline
0     & 0                  &0                  &1               \\
0.1   &0.0468822     &0.0048011    &1.1140776  \\
0.2   &0.1256645     &0.0446632    &1.2933290  \\
0.3   &0.2069145     &0.1654736    &1.5815223  \\
0.4   &0.2635181     &0.4099565    &1.9951711  \\
0.5   &0.2692106     & 0.7956850    &2.5023956  \\
0.6   &0.2064134     &1.2918180    &3.0099141  \\
0.7   &0.0771310     &1.8063493    &3.3742327  \\
0.8   &-0.087456     & 2.1991365    &3.4496062  \\
0.9   &-0.224929     &2.3316429     &3.1706707  \\
1      &-0.253600     &2.1480028     &2.6408020  \\
\hline
\end{tabular}
\label{table10} 
\end{table} 

%%%%%%%%%%%%%%%%%%%%%%%
\begin{table}[tbp]
\caption{HF solution of Example 4 ($\alpha=0.75, m=300$)}
\centering
\begin{tabular}{c c c c}
\hline
\multicolumn{1}{p{1.0cm}}{\centering $t$} & \multicolumn{1}{p{1.0cm}}{\centering $x(t)$} &\multicolumn{1}{p{1.5cm}}{\centering $y(t)$} &\multicolumn{1}{p{1.5cm}}{\centering $z(t)$} \\
\hline
0      &0                   &0              &1             \\
0.1   & 0.0220867    &0.000664  &1.1049819\\
0.2   & 0.0738818    &0.008001  &1.2359541\\
0.3   &0.1483868    &0.034736    &1.4156685\\
0.4   &0.2403191    &0.098820    &1.6654756\\
0.5   &0.3435439    & 0.222038   &2.0030821\\
0.6   &0.4503364    &0.427807     &2.4386113\\
0.7   &0.5510685    &0.737687    &2.9690832\\
0.8   &0.6342655    &1.166249    & 3.5714742\\
0.9   & 0.6871904    &1.714147    &4.1949900\\
1      &0.6972347    &2.359606     &4.7540760\\
\hline
\end{tabular}
\label{table11} 
\end{table} 
%%%%%%%%%%%%%%%%
\begin{example}\label{ex4}
Let us consider the fractional order linear differential-algebraic equations
\begin{equation} \label{GrindEQ__66_} 
{}_{0}^{C} D_{t}^{\alpha } x\left(t\right)-t^{2} x\left(t\right)+y\left(t\right)-2t=0, t\in \left[0,1\right], x\left(0\right)=0,            
\end{equation} 
\begin{equation} \label{GrindEQ__67_} 
{}_{0}^{C} D_{t}^{\alpha } y\left(t\right)-2z\left(t\right)+2(t+1)=0, \alpha \in \left(0,1\right], y\left(0\right)=0,            
\end{equation} 
\begin{equation} \label{GrindEQ__68_} 
z\left(t\right)-y\left(t\right)-2tx\left(t\right)+t^{4} -t-1=0, z\left(0\right)=1.             
\end{equation} 
In case of $\alpha =1$, the exact solution is $x\left(t\right)=t^{2} $, $y\left(t\right)=t^{4} $, $z\left(t\right)=2t^{3} +t+1$.
\end{example}
The piecewise linear approximate solutions (Tables \ref{table9} to \ref{table11}) produced by the proposed method match the semi-analytical solutions via fractional differential transform method and homotopy analysis method in \cite{ref32} (see Example 4 in \cite{ref32}). The results show that our numerical method is not only accurate but also computationally attractive (Tables \ref{table7} and \ref{table8}).
%%%%%%%%%%%%%%%%
\begin{example}\label{ex5}
The fractional order version of Chemical Akzo Nobel problem is 
\begin{equation} \label{GrindEQ__69_} 
{}_{0}^{C} D_{t}^{\alpha } y_{1} \left(t\right)=-2k_{1} y_{1}^{4} \left(t\right)y_{2}^{0.5} \left(t\right)+k_{2} y_{3} \left(t\right)y_{4} \left(t\right)-\frac{k_{2} }{K} y_{1} \left(t\right)y_{5} \left(t\right)-B09,         
\end{equation} 
\begin{equation} \label{GrindEQ__70_} 
{}_{0}^{C} D_{t}^{\alpha } y_{2} \left(t\right)=-0.5k_{1} y_{1}^{4} \left(t\right)y_{2}^{0.5} \left(t\right)-k_{3} y_{1} \left(t\right)y_{4}^{2} \left(t\right)-0.5k_{4} y_{6}^{2} \left(t\right)y_{2}^{0.5} \left(t\right)+ B010,   
\end{equation} 
\begin{equation} \label{GrindEQ__71_} 
{}_{0}^{C} D_{t}^{\alpha } y_{3} \left(t\right)=k_{1} y_{1}^{4} \left(t\right)y_{2}^{0.5} \left(t\right)-k_{2} y_{3} \left(t\right)y_{4} \left(t\right)+\frac{k_{2} }{K} y_{1} \left(t\right)y_{5} \left(t\right),           
\end{equation} 
\begin{equation} \label{GrindEQ__72_} 
{}_{0}^{C} D_{t}^{\alpha } y_{4} \left(t\right)=-k_{2} y_{3} \left(t\right)y_{4} \left(t\right)+\frac{k_{2} }{K} y_{1} \left(t\right)y_{5} \left(t\right)-2k_{3} y_{1} \left(t\right)y_{4}^{2} \left(t\right),           
\end{equation} 
\begin{equation} \label{GrindEQ__73_} 
{}_{0}^{C} D_{t}^{\alpha } y_{5} \left(t\right)=k_{2} y_{3} \left(t\right)y_{4} \left(t\right)-\frac{k_{2} }{K} y_{1} \left(t\right)y_{5} \left(t\right)+k_{4} y_{6}^{2} \left(t\right)y_{2}^{0.5} \left(t\right),           
\end{equation} 
\begin{equation} \label{GrindEQ__74_} 
0=K_{S} y_{1} \left(t\right)y_{4} \left(t\right)-y_{6} \left(t\right),            
\end{equation} 
where $B09=k_{3} y_{1} \left(t\right)y_{4}^{2} \left(t\right)$, $B010=klA\left(\left(p\left(CO_{2} \right) H\right)-y_{2} \left(t\right)\right)$.\\
The values of parameters are\\
$k_{1} =18.7, k_{2} =0.58, k_{3} =0.09, k_{4} =0.42, K_{S} =115.83, K=34.4, klA=3.3$, \\ $H=737, p\left(CO_{2} \right)=0.9$,\\
and the initial values are \\
$y_{1} \left(0\right)=0.444, y_{2} \left(0\right)=0.00123,y_{3} \left(0\right)=0, y_{4} \left(0\right)=0.007,y_{5} \left(0\right)=0$,\\
$y_{6} \left(0\right)=K_{S} \times 0.444\times 0.007$.
\end{example}
The proposed TFs based numerical method is applied to the Chemical Akzo Nobel problem (CANP). The step size of 1/200 is used for all computations. The numerical solution by modified Rosenbrock method of order 2, MRM2 \cite{ref38} is too determined to authenticate that the piecewise linear HF solution by the proposed numerical method converges to the original solution (i.e. numerical solution obtained by MRM2) when $\alpha $ equals 1. In addition to MRM2, Adomian decomposition method (ADM), fractional differential transform method with Adomian polynomials (FDTM) \cite{ref39} and homotopy analysis method (HAM) are employed to solve CANP and the respective approximate solutions are plotted in Figures \ref{Figure1} to \ref{Figure6}. The semi-analytical techniques; ADM, FDTM and HAM exhibited numerical instability and failed to approximate the solution of CANP in both the integer order case and non-integer order case. It is seen from Figures \ref{Figure7} and \ref{Figure8} that, in case of $\alpha =1$, the HF solution is in good comply with the solution obtained by MRM2 and the fractional order HF solutions approach the integer order solution in the limit of $\alpha $ tends to 1. Therefore, the proposed HF based numerical method is so powerful that it can handle even highly nonlinear, high dimensional and stiff differential-algebraic equations of arbitrary order. The time elapsed during each computation is recorded and tabulated in Table \ref{table7}. In both cases (integer and fractional order), the proposed method needs higher CPU usage but such higher computational time is justified to get acceptable approximate solutions to such a high dimensional and stiff system.

%%%%%%%%%%%%
\begin{figure}
\centering
 \includegraphics[scale=0.7 ]{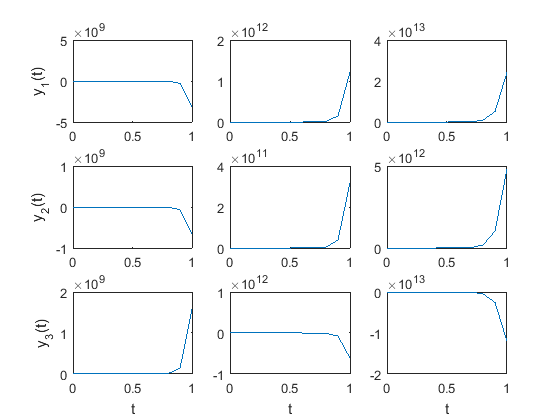}
 \caption{ADM solution of Chemical Akzo Nobel problem for $\alpha =1$(first column), $\alpha =0.8$(second column), $\alpha =0.6$(third column)}
 \label{Figure1}
 \end{figure}
%%%%%%%%%%%

\begin{figure}
\centering
 \includegraphics[scale=0.7 ]{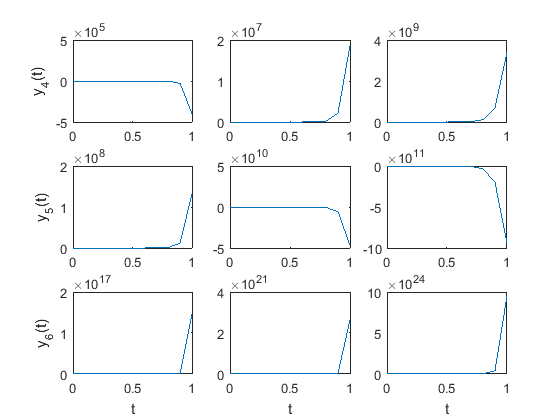}
 \caption{ADM solution of Chemical Akzo Nobel problem for $\alpha =1$(first column), $\alpha =0.8$(second column), $\alpha =0.6$(third column)}
 \label{Figure2}
 \end{figure}
%%%%%%%%%%%%%%
\begin{figure}
\centering
 \includegraphics[scale=0.75]{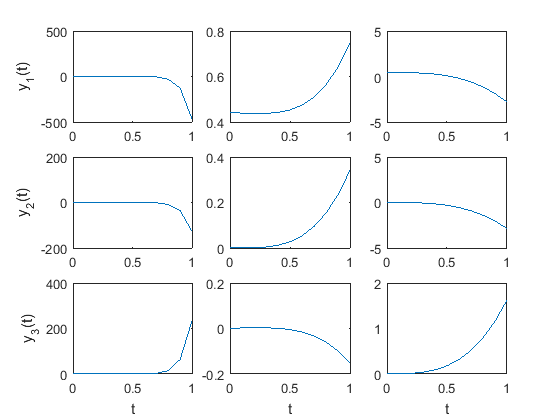}
 \caption{ FDTM solution of Chemical Akzo Nobel problem for $\alpha =1$(first column), $\alpha =0.8$(second column), $\alpha =0.6$(third column)}
 \label{Figure3}
 \end{figure}
%%%%%%%%%%%%%%

\begin{figure}
\centering
 \includegraphics[scale=0.7 ]{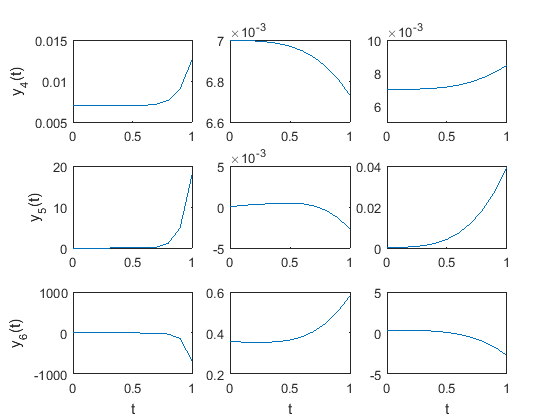}
 \caption{FDTM solution of Chemical Akzo Nobel problem for $\alpha =1$(first column), $\alpha =0.8$(second column), $\alpha =0.6$(third column)}
 \label{Figure4}
 \end{figure}
%%%%%%%%%%%%%%%%%

\begin{figure}
\centering
 \includegraphics[scale=0.75 ]{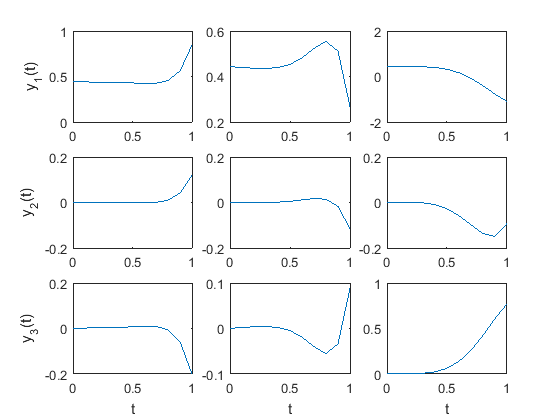}
 \caption{HAM solution of Chemical Akzo Nobel problem for $\alpha =1$(first column), $\alpha =0.8$(second column), $\alpha =0.6$(third column)}
 \label{Figure5}
 \end{figure}

%%%%%%%%%%%

\begin{figure}
\centering
 \includegraphics[scale=0.7 ]{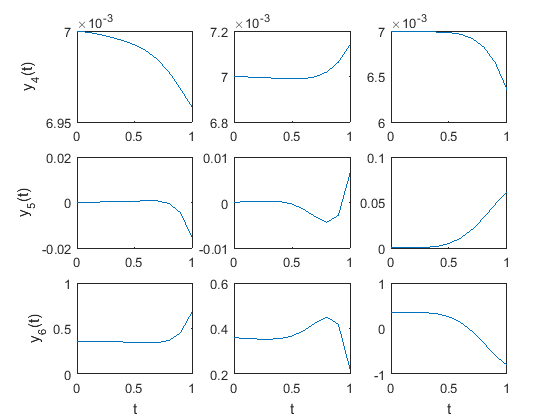}
 \caption{HAM solution of Chemical Akzo Nobel problem for $\alpha =1$(first column), $\alpha =0.8$(second column), $\alpha =0.6$(third column)}
 \label{Figure6}
 \end{figure}

%%%%%%%%%%%

\begin{figure}
\centering
 \includegraphics[scale=0.37 ]{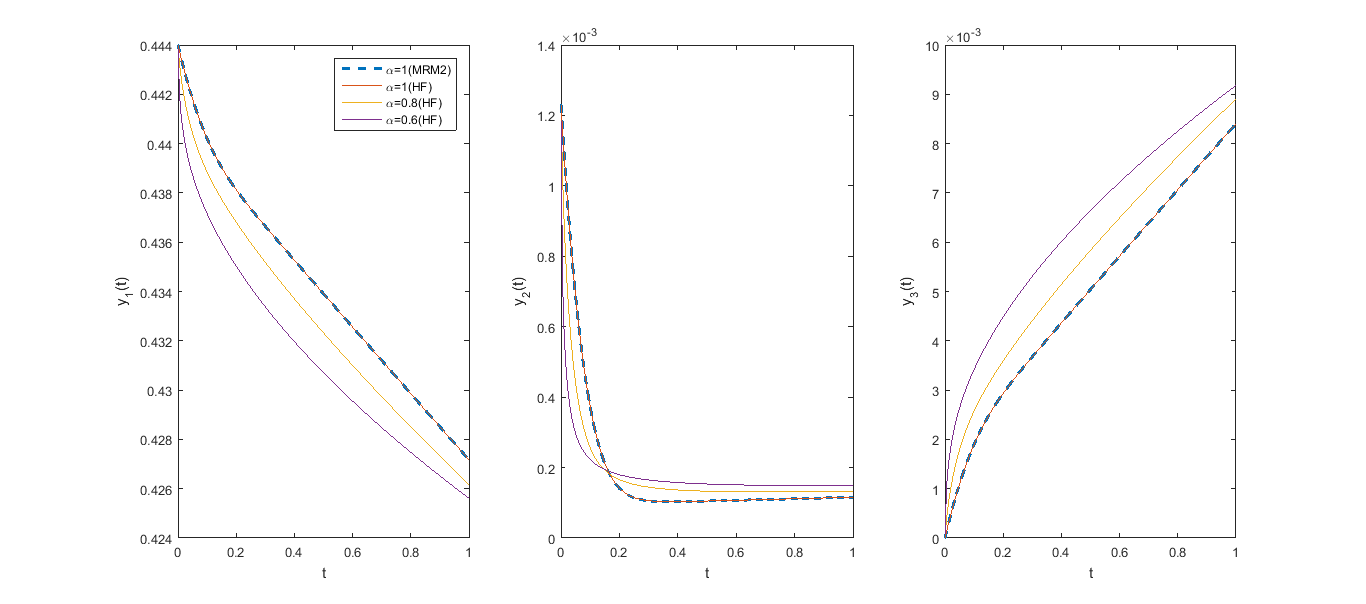}
 \caption{Comparison of HF solutions of Chemical Akzo Nobel problem for different$\alpha $}
 \label{Figure7}
 \end{figure}

%%%%%%%%%%%

\begin{figure}
\centering
 \includegraphics[scale=0.35]{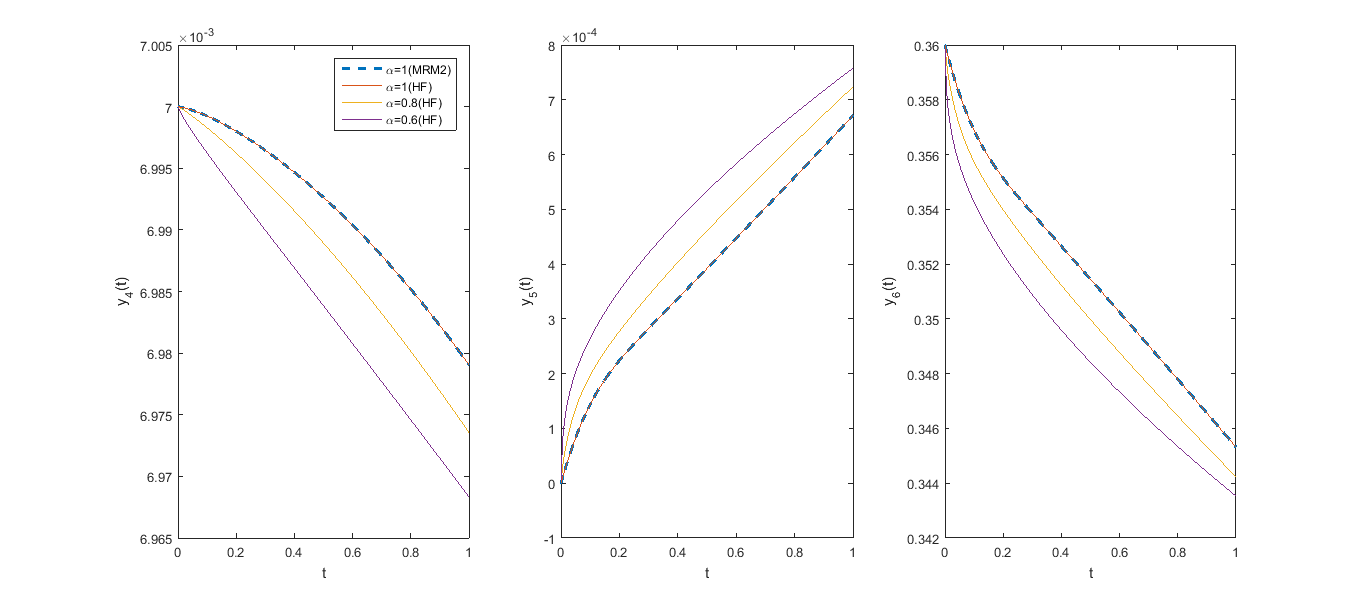}
 \caption{Comparison of HF solutions of Chemical Akzo Nobel problem for different$\alpha $}
 \label{Figure8}
 \end{figure}
%%%%%%%%%%%%%%%%%%%%%%%%%%%%%%%%%
\section{Concluding remarks}\label{sec7}
\noindent We have the following conclusions.
{\renewcommand{\labelitemi}{$\triangleright$}
\begin{itemize}
\item The derived HF estimates for the Riemann-Liouville fractional order integral worked well and can be used to approximate integration of any order.
\item The proposed numerical method is capable enough to solve numerically a wide variety (stiff and non-stiff) of differential-algebraic equations of arbitrary order.
\end{itemize}}
%%%%%%%%%%%%%%%%

\end{document}